\documentclass[12pt,centertags,oneside]{amsart}
\usepackage{amsmath,amstext,amsthm,amscd,typearea,hyperref}
\usepackage{amssymb}
\usepackage{a4wide}
\usepackage[mathscr]{eucal}
\usepackage{mathrsfs}
\usepackage{typearea}
\usepackage{charter}
\usepackage{pdfsync}
\usepackage[a4paper,width=16.2cm,top=3cm,bottom=3cm]{geometry}

\numberwithin{equation}{section}

\newtheorem{theorem}{Theorem}[section]

\newtheorem{lemma}[theorem]{Lemma}

\newcommand{\cali}[1]{\mathscr{#1}}

\newcommand{\dist}{\mathop{\mathrm{dist}}\nolimits}

\newcommand{\vol}{\mathop{\mathrm{vol}}}

\newcommand{\ddc}{dd^c}

\newcommand{\DSH}{{\rm DSH}}

\newcommand{\Cc}{\cali{C}}

\newcommand{\Fc}{\cali{F}}

\newcommand{\B}{\mathbb{B}}
\newcommand{\C}{\mathbb{C}}

\newcommand{\R}{\mathbb{R}}

\newcommand{\norm}[1]{\lVert#1\rVert}

\emergencystretch=\maxdimen
\title[Equilibrium measure of  meromorphic correspondence]{Regularity of the equilibrium measure for meromorphic correspondences}

\author{Tien-Cuong Dinh and Hao Wu}
\address{Department of Mathematics, National University 
of Singapore, 10 Lower Kent Ridge Road, Singapore 119076.}
\email{matdtc@nus.edu.sg; e0011551@u.nus.edu}

\thanks{This work is supported by 
	the NUS and MOE grants A-0004285-00-00 and MOE-T2EP20120-0010.}

\date{}
\begin{document}

\begin{abstract}
Let $f$ be a  meromorphic correspondence on a compact K\"ahler manifold $X$ of dimension $k$. Assume that its topological degree is larger than the  dynamical degree of order $k-1$. We obtain a quantitative  regularity of the equilibrium measure of $f$ in terms of its super-potentials.
\end{abstract}

\maketitle

\medskip

\noindent {\bf Classification AMS 2020}: 32H, 37F05, 32U40.

\medskip


\medskip


\section{Introduction}\label{sec-intro}

Let $X$ be a compact K\"ahler manifold of dimension $k$ endowed with a K\"ahler form $\omega$ normalized so the the integral of $\omega^k$ on $X$ is 1. 
A \textit{meromorphic correspondence} $f$ on $X$ is given by a finite holomorphic chain $\Gamma:=\sum_{j=1}^m \Gamma_j$ such that 
\begin{enumerate}
	\item  for each $j$, $\Gamma_j$ is an irreducible analytic subset (may not be distinct) of dimension $k$ in $X\times X$;
	
	\item  the natural projections $\pi_i:X\times X\to X$, $i=1,2$, restricted to each $\Gamma_j$, are surjective and generically finite.
\end{enumerate}
We call $\Gamma$ the \textit{graph} of $f$. The action of $f$ is given by 
$$f(A):=\pi_2\big( \pi_1^{-1}(A) \cap \Gamma\big) \quad \text{and}\quad  f^{-1}(A):=\pi_1\big( \pi_2^{-1}(A)\cap\Gamma\big) \quad\text{for every set}\quad A\subset X.$$ 

We see that a meromorphic correspondence can be viewed as a meromorphic multi-valued map from $X$ to itself.  We call the two sets
$$I_1(f):=\{x\in X:\, \dim \pi_1^{-1}(x) \cap \Gamma>0\} \quad \text{and} \quad I_2(f):=\{x\in X:\, \dim \pi_2^{-1}(x)\cap \Gamma >0\} $$
the \textit{first and second indeterminacy sets} of $f$. They are of codimension at least $2$ in $X$. If $I_1(f)$ is empty, then $f$ is a \textit{holomorphic correspondence}. 
If for a generic point $x\in X$, $\pi_1^{-1}(x)$ contains only one point and $\Gamma$ is irreducible, then $f$ is a dominant meromorphic map.

To define $f^n:=f\circ \cdots \circ f$ ($n$ times), we compose $f$ with itself outside an analytic set and compactify the graph in order to get the graph of $f^n$. We can also define the push-forward and pull-back of smooth forms under the actions of $f^n$, see Section \ref{sec-pull}.
For $0\leq p\leq k$, it is known that the limit 
$$d_p:=\lim_{n\to\infty} \|(f^n)^*(\omega^p)\|^{1/n}$$
exists and is a bi-meromorphic invariant, see \cite{dinh-sibony-CMH}. We call it the \textit{dynamical degree of order $p$} of $f$. We have that $d_0$ is the number of points in a generic fiber of the map $\pi_1 |_\Gamma$ and $d_k$ is the one of $\pi_2|_\Gamma$ (counting multiplicity). We also call $d_k$ the \textit{topological degree} of $f$.  For more details on dynamical degrees and K\"ahler geometry, the readers may refer to \cite{Demailly-regular,DNT-ind,Ma-book,Truong}.
\medskip

Sibony and the first author \cite{dinh-sibony-CMH} proved that if $f$ is a meromorphic correspondence on $X$ with a dominant topology degree, i.e., $d_k>d_{k-1}$, then $f$ admits an \textit{equilibrium measure}. More precisely, there exists a unique probability measure $\mu$ on $X$ such that $d_k^{-n}(f^n)^*(\omega^k)$ converge to $\mu$ weakly and it satisfies  $ f^* (\mu)=d_k\mu$, see also \cite{guedj}.

\medskip
Recall that a \textit{quasi-plurisubharmonic} (\textit{quasi-p.s.h.}\ for short) function on $X$ is locally the difference of a plurisubharmonic function and a smooth one. A subset of $X$ is called \textit{pluripolar} if it is contained in $\{\phi=-\infty\}$ for some  quasi-plurisubharmonic function $\phi$, see also \cite{Vu-tams}. The following two sets are compact in $L^p(X)$ for $1\leq p<\infty$,
$$\{  \phi \, \text{ quasi-p.s.h.}: \, \ddc\phi\geq -\omega  ,\quad \max \phi=0\}  $$
$$\Big\{ \phi \, \text{ quasi-p.s.h.} : \, \ddc\phi\geq -\omega  , \quad \int_X \phi \,\omega =0 \Big\}. $$

Let $\DSH(X)$ denote the space of all d.s.h.\ functions on $X$, i.e., functions which are equal to a difference of two quasi-p.s.h.\  functions. We identify two d.s.h.\ functions if they are equal outside a pluripolar set. If $\phi$ is d.s.h., there are positive closed $(1,1)$-currents $T^\pm$ on $X$ such that $\ddc\phi=T^+-T^-$.
Since these two currents are cohomologous, they have the same mass. We define a norm on $\DSH(X)$ by
$$\|\phi\|_\DSH:=\Big|\int_X \phi \,\omega \Big| + \inf\|T^\pm\|,$$
where the infimum is taken over all $T^\pm$ as above. This norm induces a distance on $\DSH(X)$. But in this article, we will use the $L^1$-distance on this space.

For every d.s.h.\ function $\phi$ on $X$, there are quasi-p.s.h.\ functions $\phi^\pm$ and a number $m$, such that $\phi=\phi^+-\phi^- +m$, $\max \phi^\pm=0$, $\ddc \phi^\pm \geq -c\norm{\phi}_{\DSH}\omega$, $\norm{\phi^\pm}_{\DSH}\leq c\norm{\phi}_{\DSH}$ and $|m|\leq c\norm{\phi}_{\DSH}$, where $c>0$ is a constant independent of $\phi$, see \cite{dinh-sibony-CMH}.
\medskip

The following is our main result, which is already new when $f$ is a dominant meromorphic map or a correspondence on a Riemann surface. 

\begin{theorem}\label{main-theorem}
Let $f$ be a  meromorphic correspondence on a compact K\"ahler manifold $(X,\omega)$ of dimension $k$. Let $d_p$ denote its dynamical degree of order $p$. Assume that $d_k>d_{k-1}$. Let $\mu$ denote the equilibrium measure of $f$. Then there is a number $\gamma>0$ such that if $\Fc\subset \DSH(X)$ is a bounded subset of $\DSH(X)$ we have for all $\phi,\psi\in \Fc$,
$$\Big|\int_X \phi \,d\mu-\int_X \psi \,d\mu\Big|\leq c \big(\log^\star \norm{\phi-\psi}_{L^1}\big)^{-\gamma},$$
where $c=c(\Fc)>0$ is a constant independent of $\phi$ and $\psi$ and $\log^\star(t):=1+|\log t|$.
\end{theorem}

Recall that the functional $\phi\mapsto \int_X \phi \,d\mu$ with $\phi$ d.s.h.\ is called the \textit{super-potential} of $\mu$. The above theorem says that the super-potential of the equilibrium measure is $\log^\gamma$-continuous in the above sense.  It is known that $\log^\gamma$-continuity is useful in complex dynamics, see \cite{BD-state,DKW:LLT}. For $k=1$, the result says that $\mu$ has locally $\log^\gamma$-continuous potentials.

 In Section \ref{sec-pull}, we will discuss the action of $f$ on d.s.h.\ functions and positive closed currents. The proof of Theorem \ref{main-theorem} will be given in Section \ref{sec-proof}.

\section{Pull-back of d.s.h.\ functions and $(1,1)$-currents} \label{sec-pull}

We refer to \cite{Demailly-regular,dinh-sibony-regular,dinh-sibony-CMH,pull-back} for properties discussed in this section.

Let $f$ be a meromorphic correspondence on $X$ which is given by a chain $\Gamma$ as above.
Since $\pi_1$ and $\pi_2$ are holomorphic proper submersions, $(\pi_1)_*$ and  $\pi_2^*$ are well-defined on the space of currents. So for a current $T$ on $X$,
we define the pull-back $f^*(T)$ by 
$$f^*(T):=(\pi_1)_*\big( \pi_2^*(T)\wedge [\Gamma]\big)$$
if $\pi_2^*(T)\wedge [\Gamma]$ is well-defined. Here, $[\Gamma]$ is the positive closed current of integration on $\Gamma$. In particular, if $T$ is a smooth form, then $f^*(T)$ is well-defined and it is an $L^1$-form.

Let $T$ be a smooth $(p,q)$-form. It is not hard to see that $f^*(T)$ is closed if $T$ is closed, and $f^*(T)$ is exact if $T$ is exact. Therefore, $f^*$ induces a linear map  on Hodge cohomology
 $$f^* : H^{p,q}(X,\C)\longrightarrow  H^{p,q}(X,\C).$$
 In general, we do not have $(g^n)^*\circ (g^m)^*\neq (g^{n+m})^*$. When the equality holds for $p=q=1$, $f$ is said to be \textit{algebraically stable} in the sense of Fornaess-Sibony \cite{For-Sib,sibony-panor}.

\medskip

For a positive  $(p,p)$-current $T$ on the compact K\"ahler manifold $(X,\omega)$ of dimension $k$, recall that the mass of $T$ is defined by $\norm{T}:=\int_X T\wedge \omega^{k-p}$. If $T$ is closed, the mass of $T$ only depends on the class of $T$ in $H^{p,p}(X,\C)$.

By the above definition, we have 
$$\norm{f^*(\omega^p)}=\int_X f^*(\omega^p)\wedge \omega^{k-p}=\int_{\Gamma} \pi_2^*(\omega^p)\wedge \pi_1^*(\omega^{k-p})=\int_X f_*(\omega^{k-p})\wedge \omega^p=\norm{f_*(\omega^{k-p})}.$$
Observe that if $S_1$ and $S_2$ are two positive closed $(p,p)$-forms in the same cohomology class, then $\norm{f^*(S_1)}=\norm{f^*(S_2)}$. Thus, the norm of $f^*$ acting on $H^{p,p}(X,\C)$ is related to the action $f^*$ on positive closed currents.
Since $H^{0,0}(X,\C)\simeq \C$ and  $H^{k,k}(X,\C) \simeq \C$, $f^*$ acting on $H^{0,0}(X,\C)$ is the multiplication by $d_0$ and its action on  $H^{k,k}(X,\C)$ is the multiplication by $d_k$. Here, $d_0$ and $d_k$ are the dynamic degrees defined in the introduction.
\medskip

For simplicity, in what follows, we assume that $\Gamma$ is irreducible. The general case can be deduced from this case.
 Let $\pi:\widetilde \Gamma\to \Gamma \subset X^2$ be a desingularization of $\Gamma$, which can be obtained by a finite number of blow-ups of $X^2$ over the singularities of $\Gamma$. By Blanchard \cite{Blanchard}, $\widetilde \Gamma$ is a compact K\"ahler manifold. Define $\Pi_i:=\pi_i |_\Gamma \circ \pi$. 
Then $(\Pi_1)_*:\DSH(\widetilde \Gamma)\to \DSH(X)$ and $\Pi_2^*:\DSH(X)\to \DSH(\widetilde \Gamma)$ are well-defined, as well as $f^*=(\Pi_1)_*\circ \Pi_2^*$  on  $\DSH(X)$. In particular, we can define $\pi_2^* (\phi) \wedge [\Gamma]:=\pi_*( \Pi_2^*( \phi) )$ for $\phi\in \DSH(X)$. This definition is independent of the choice of $\pi$.

Now let  $T$ be a positive closed $(1,1)$-current. We can take a real smooth $(1,1)$-form $S$ which is in the same cohomology class of $T$. Then there exists a d.s.h.\ function $\phi$ such that $\ddc \phi=T-S$ and we have 
$$ \pi_2^*(T)\wedge [\Gamma]=\pi_2^*(\ddc \phi +S)\wedge [\Gamma]=\pi_2^*(\ddc \phi)\wedge [\Gamma] +\pi_2^*(S)\wedge [\Gamma] =\ddc(  \pi_2^*(\phi)  [\Gamma]) +  \pi_2^*(S)\wedge [\Gamma]. $$
In the last sum, the first term is well-defined from our discussion above, and the second one is also well-defined since $S$ is smooth. Hence $\pi_2^*(T)\wedge [\Gamma]$ is well-defined. So we can define 
$$f^*(T):=(\pi_1)_*\big( \ddc(  \pi_2^*\phi \wedge  [\Gamma]) +  \pi_2^*(S)\wedge [\Gamma] \big).  $$
 Clearly, this definition is independent of the choices of $\phi$ and $S$. It depends continuously on $T$ and we have $\norm{f^*(T)}=\norm{f^*(S)}$.
 
Similarly, we can define the push-forward $f_*$ on d.s.h.\ functions and positive closed $(1,1)$-currents as above.
\medskip

Since $\Pi_1$ is holomorphic and proper, it is not hard to see  that the push-forward operator $(\Pi_1)_*:\DSH(X)\to \DSH(\widetilde\Gamma)$ is bounded with respect to the $\DSH$-norm.
The pull-back operator $\Pi_2^*:\DSH(\widetilde \Gamma)\to \DSH(X)$ is bounded as well by the next lemma.

\begin{lemma}\label{pull-back-bounded}
	Let $g:X\to Y$ be a surjective holomorphic map between two compact K\"ahler manifolds.  Then the pull-back operator $g^*:\DSH(Y)\to \DSH(X)$ is bounded with respect to the $\DSH$-norm.
\end{lemma}

\begin{proof}
	Let $\phi\in \DSH(Y)$ with $\norm{\phi}_{\DSH}=1$. It enough to show that $g^*(\phi)$ is contained in a compact subset of $\DSH(X)$. We write $\phi=\phi^+-\phi^-+m$ for quasi-p.s.h.\ functions $\phi^\pm$, where  $\max \phi^\pm=0$, $\ddc \phi^\pm \geq -c\omega$, $\norm{\phi^\pm}_{\DSH}\leq c$ and $|m|\leq c$, for some constant $c>0$  independent of $\phi$. We have $g^*(\phi)=g^*(\phi^+)-g^*(\phi^-) +m$. Moreover, $\ddc g^*(\phi^\pm)\geq -c g^*(\omega)$ and   $\max g^*(\phi^\pm)=0$. 
Therefore, $g^*(\phi^\pm)$ are in a compact subset of $L^1(X)$. Combining with $|m|\leq c$, we deduce that $g^*(\phi)$ is also in a compact subset of $\DSH(X)$.
\end{proof}

\section{Proof of the main theorem}\label{sec-proof}

We start the proof of Theorem \ref{main-theorem} with the following lemma.

\begin{lemma}\label{log-continuous}
Let $B$ be a metric space of finite diameter. Let $u:B\to \R$ be a Lipschitz function with Lipschitz constant $c_1$, and let $\Lambda:B\to B$ be a $\theta$-H\"older map ($0<\theta\leq 1$) with $\norm{\Lambda}_{\Cc^{0,\theta}}=c_2$. Let $\delta>1$ be a constant. Then the function 
$$v(x):=\sum_{n=0}^\infty \delta^{-n} u(\Lambda^n(x))$$
is $\log^\gamma$-continuous for some constant $\gamma>0$. That is, we have 
$$|v(x)-v(y)|\leq c \big(\log^\star\dist(x,y)\big)^{-\gamma} \quad\text{for all}\quad x,y\in B,$$
where $c>0$ is a constant depending on $c_1,c_2,\theta$ and $\delta$.
\end{lemma}
\proof
When $\theta=1$, by \cite[Lemma 1.19]{dinh-sibony:cime}, $v$ is H\"older continuous, which clearly implies the desired inequality. Now we assume $\theta< 1$.
Define $\xi:=\dist(x,y)$. We only need to consider $\xi<1/4$ because otherwise the desired inequality holds for $c$ large enough.      We have $1+|\log\xi|\leq 2 |\log \xi|$. So it is enough to prove
\begin{equation}\label{vx-vy}
|v(x)-v(y)|\leq c |\log\xi|^{-\gamma}  \quad\text{for some}\quad c>0.
\end{equation}
\vskip 5pt

Fix a constant $c_3>1$ larger than $c_2$. By induction,
 we have for $n\geq 1$,
 \begin{align*}
 \big|\Lambda^n(x)-\Lambda^n(y)\big|\leq c_2^{1+\theta+\cdots+\theta^{n-1}}\xi^{\theta^n}\leq c_3^{1/( 1-\theta)}\xi^{\theta^n}.
\end{align*}

After multiplying $u$ by a constant, one can assume $\max u-\min u\leq 1$ and $c_1\leq 1$. 
 Therefore, using that $\xi^{\theta^n}\leq \xi^{\theta^m}$ for $n\leq m$, we get 
 \begin{align*}
	&|v(x)-v(y)|\leq\sum_{n=0}^\infty \delta^{-n}\big|u(\Lambda^n(x))-u(\Lambda^n(y))  \big|\\
	&\leq |u(x)-u(y)|+ \sum_{n=1}^N \delta^{-n} \big|u(\Lambda^n(x))-u(\Lambda^n(y))\big| + \sum_{n=N+1}^\infty \delta^{-n} \\
	&\leq \xi+ \sum_{n=1}^N \delta^{-n} c_3^{1/( 1-\theta)}\xi^{\theta^n}+ {\delta^{-N}\over \delta-1}\leq C_{1} (\xi^{\theta^N} +\delta^{-N}).
\end{align*}
Here, $N$ is any positive integer,  and $ C_{1}$ is some constant depending on $c_3, \theta$ and $\delta$. 

Now we take $$N:=\lfloor \log|\log \xi|/|2\log\theta| \rfloor \quad\text{and} \quad\gamma:=\log\delta/|2\log\theta|.$$
 Recall that $\theta<1$ and  $\xi<1/4$. So
\begin{align*}
\xi^{\theta^N}\leq \xi^{\theta^{ \log|\log \xi|/|2\log\theta|}} =\exp\Big[ \theta^{ \log|\log \xi|/|2\log\theta|} \log\xi   \Big]=\exp\Big[  |\log\xi|^{-1/2}\log\xi \Big]\\
=\exp\Big[ -|\log\xi|^{1/2} \Big] \leq C_{2} \exp\Big[-\gamma \log|\log\xi|  \Big]=C_{2} |\log\xi|^{-\gamma}
\end{align*}
for some constant $C_{2}$ depending on $\theta$ and $\delta$. For  the inequality above, we use  that $-|\log\xi|^{1/2}\leq A-\gamma\log|\log \xi|$ for some large constant $A$ depending on $\theta$ and $\delta$.

For the term $\delta^{-N}$, we have
$$\delta^{-N}\leq\delta^{-\log|\log \xi|/|2\log\theta|+1}=\delta\exp\Big[\big(-\log|\log \xi|/|2\log\theta|\big)   \cdot \log\delta     \Big]=\delta  \exp\Big[-\gamma \log|\log\xi|  \Big],$$
which is equal to $\delta  |\log\xi|^{-\gamma}$.
So we conclude that 
$$ |v(x)-v(y)|\leq C_{1} (C_{2} |\log\xi|^{-\gamma} + \delta |\log\xi|^{-\gamma}  )= C_{1}(C_{2}+\delta)|\log\xi|^{-\gamma}. $$
This gives  \eqref{vx-vy}  for $c=C_{1}(C_{2}+\delta)$  and finishes the proof of the lemma.
\endproof

From the discussions in Section \ref{sec-pull}, $f_*$ is a Lipschitz map on the space of d.s.h.\ functions with respect to the DSH-distance. Moreover, we need that with the $L^1$-distance on $\DSH(X)$, $f_*$ is H\"older on bounded subset $\DSH(X)$. This is the content of the next lemma. 

\begin{lemma}\label{theta-holder}
There is a number $0<\theta\leq 1$ such that $\phi\mapsto f_*(\phi)$ is $\theta$-H\"older with respect to the $L^1$-norm on every bounded subset of $\DSH(X)$. That is, for every bounded subset  $\Fc\subset\DSH(X)$, there exists a constant $c=c(\Fc)$ such that 
$$\big\|f_*(\phi)-f_*(\psi)\big\|_{L^1}\leq c \|\phi-\psi\|_{L^1}^\theta \quad \text{for} \quad \phi,\psi\in\Fc.$$
\end{lemma}
\proof
Observe that we can replace $\Fc$ by $\Fc-\Fc$ so that it is enough to  prove that $\norm{f_*(\phi)}_{L^1}\lesssim \norm{\phi}_{L^1}^\theta$ for $\phi\in \Fc$. We can work with each component of $\Gamma$ separately. So, we can assume that $\Gamma$ is irreducible for simplicity.

 Let $\pi:\widetilde \Gamma\to \Gamma$ be a desingularization of $\Gamma$ as in Section \ref{sec-pull} and we fix a K\"ahler form $\widetilde \omega$ on $\widetilde\Gamma$. 
Define $\Pi_i:=\pi_i\circ \pi$. 
We have $\Pi_1^*:\DSH(X)\to \DSH(\widetilde \Gamma)$ and $(\Pi_2)_*:\DSH(\widetilde \Gamma)\to \DSH(X)$ and $f_*=(\Pi_2)_*\circ \Pi_1^*$.  From  Section \ref{sec-pull}, we know that all these maps  are well-defined and bounded with respect to the $\DSH$-norm.

\vskip 5pt

We first consider $(\Pi_2)_*$. Since $\Pi_2$ is  holomorphic, $(\Pi_2)^*(\omega^j)$ is smooth and hence $(\Pi_2)^*(\omega^j) \lesssim\widetilde\omega^j$. For any $\varphi\in \DSH(\widetilde\Gamma)$, we have 
$$\norm{(\Pi_2)_*(\varphi)}_{L^1} \leq\int_{X} (\Pi_2)_* (|\varphi|)\, \omega^k  = \int_{\widetilde\Gamma} |\varphi| \,(\Pi_2)^*(\omega^k) \lesssim \int_{\widetilde\Gamma} |\varphi| \,\widetilde\omega^k=\norm{\varphi}_{L^1}.$$

It remains to prove  $\|\Pi_1^*(\phi)\|_{L^1}\lesssim \|\phi\|_{L^1}^\theta$ for $\phi\in \Fc$.  Let $\Sigma$ denote the set of critical values of $\Pi_1$ and let $\widetilde\Sigma:=\Pi_1^{-1}(\Sigma)$. We also denote by $\widetilde\Sigma'$ the set of critical points of $\Pi_1$, which is the zero set $\{\mathrm{Jac} \Pi_1=0\}$ of the Jacobian of $\Pi_1$ in local holomorphic coordinates (we cover $\widetilde\Gamma$ by a finite number of local holomorphic charts). Clearly, $\widetilde\Sigma'\subset\widetilde\Sigma$. Using blow-ups, we can assume that  $\widetilde\Sigma$ is a finite union of smooth manifolds.
Let $\widetilde\Sigma_\epsilon$ be the $\epsilon$-neighbourhood of $\widetilde \Sigma$. We separate $\|\Pi_1^*(\phi)\|_{L^1}$ into two parts:
\begin{equation}\label{pi-l1}
\|\Pi_1^*(\phi)\|_{L^1}\leq \int_{\widetilde \Gamma\setminus \widetilde\Sigma_\epsilon} \Pi_1^*(|\phi|)\, \widetilde\omega^k + \int_{\widetilde\Sigma_\epsilon} \Pi_1^*(|\phi|)\, \widetilde\omega^k .
\end{equation}

\vskip5pt
\noindent\textbf{Claim 1.}
	We have for some $\alpha>0$, independent of $\phi$,  that
	$$\int_{\widetilde \Gamma\setminus \widetilde\Sigma_\epsilon} \Pi_1^*(|\phi|)\, \widetilde\omega^k \lesssim \epsilon^{-\alpha} \norm{\phi}_{L^1} \quad\text{for all} \quad \phi\in L^1(X).  $$

\proof[Proof of Claim 1]
Since $\widetilde\Sigma'$ is an analytic subset of $\widetilde\Gamma$ defined by $\mathrm{Jac}\Pi_1$, we can apply Lojasiewicz inequality \cite[Theorem 6.4]{semianalytic} to the function $\mathrm{Jac}\Pi_1$ and  get
$$\dist(x,\widetilde\Sigma')\lesssim  |\mathrm{Jac} \Pi_1 (x) |^A  \quad\text{for}\quad  x\in \widetilde \Gamma,$$
 where
$A>0$ is a   constant independent of $x$. Hence 
$$\dist(x,\widetilde\Sigma)\leq \dist(x,\widetilde\Sigma')  \lesssim |\mathrm{Jac} \Pi_1 (x) |^A.$$ So for $x\in \widetilde \Gamma\setminus \widetilde\Sigma_\epsilon$, one has $$|\mathrm{Jac} \Pi_1 (x) |^{-1}\lesssim \epsilon^{-1/A}.$$

For every $x\in \widetilde \Gamma\setminus \widetilde\Sigma_{\epsilon}$, one can find a small neighborhood $D\subset\widetilde \Gamma$ of $x$ such that $\Pi_1$ is injective on $D$ and hence, $\Pi_1$ is a biholomorphic map between $D$ and $\Pi_1(D)$. Therefore, at the point $\Pi_1(x)$, we have
$$(\Pi_1|_D)_*(\widetilde \omega^k)\lesssim |\mathrm{Jac} \Pi_1|^{-2} \omega^k \lesssim \epsilon^{-2/A}\omega^k.$$

Notice that $\Pi_1$ is a  holomorphic covering of degree $d_0$ outside $\widetilde\Sigma_{\epsilon}$. So 
$$(\Pi_1)_*(\widetilde \omega^k) \lesssim \epsilon^{-2/A}\omega^k  \quad\text{on}\quad \Pi_1(\widetilde\Gamma\setminus\widetilde\Sigma_{\epsilon}).$$
Taking $\alpha:=2/A$, we have
	$$\int_{\widetilde \Gamma\setminus \widetilde\Sigma_\epsilon} \Pi_1^*(|\phi|)\, \widetilde\omega^k =\int_{\Pi_1(\widetilde\Gamma\setminus\widetilde\Sigma_{\epsilon})} |\phi|\, (\Pi_1)_*(\widetilde \omega^k) \lesssim \int_X |\phi|\epsilon^{-\alpha}\,\omega^k   =\epsilon^{-\alpha} \norm{\phi}_{L^1}.  $$
	The proof of Claim 1 is finished.
\endproof

Recall that a positive measure $\nu$ on $X$ is said to be \textit{moderate} if $\nu$ has no mass on pluripolar sets and for any bounded family $\Fc$ of d.s.h.\ functions on $X$, there are constants $\alpha'>0$ and $c'>0$ such that 
$$\int_{X}e^{\alpha' |\psi|}d\nu\leq c'  \quad\text{for}\quad   \psi\in\Fc.$$
Equivalently, there are constants $\beta'>0$ and $c'>0$ such that 
$$\nu\{z\in X:|\psi(z)|>M\}\leq c'e^{-\beta' M}$$ for $M\geq 0$ and $\psi\in\Fc$.
The probability measure $\omega^k$  is moderate \cite{hormander}. The Lebesgue measure on a generic real submanifold of $X$ with dimension $\geq k$ is also moderate \cite{viet}. See also \cite{Kau-mic} for a local situation.

\vskip5pt
\noindent\textbf{Claim 2.}
	We have
	$$\int_{\widetilde\Sigma_\epsilon} \Pi_1^*(|\phi|)\, \widetilde\omega^k \lesssim \epsilon^{1/2}  \quad\text{for}\quad \phi\in \Fc.$$

\proof[Proof of Claim 2]
It is enough to consider $\epsilon$ small.
We first estimate the volume of $\widetilde\Sigma_\epsilon$. Recall that  $\widetilde\Sigma$ is a finite union of smooth manifolds.   So the volume of $\widetilde\Sigma_\epsilon$ is $\lesssim \epsilon$.
 Moreover, since $\widetilde\Sigma_\epsilon$ contains a ball of radius $\epsilon$, its volume is $\gtrsim\epsilon^{2k}$.

By Lemma \ref{pull-back-bounded}, $\Pi_1^*(|\phi|)$ is in a compact subset of $\DSH(\widetilde\Gamma)$ for $\phi\in \Fc$.
 Using the moderate property of $\widetilde\omega^k$ again, for  $\phi\in \Fc$, one has  $$\int_{\widetilde\Gamma}e^{\alpha''\Pi_1^*(|\phi|)}\,\widetilde\omega^k\leq c''$$ for some positive constants $\alpha''$ and $c''$ independent of $\phi$. In particular, $$\int_{\widetilde\Sigma_\epsilon}e^{\alpha''\Pi_1^*(|\phi|)}\,\widetilde\omega^k\leq c'' \quad\text{for all}\quad \phi\in \Fc.$$ 
 Applying Jensen's inequality to the probability measure $\widetilde\omega^k/\vol(\widetilde\Sigma_\epsilon)$ on $\widetilde\Sigma_\epsilon$, yields 
 $$\exp\Big[ \int_{\widetilde\Sigma_\epsilon} \alpha''\Pi_1^*(|\phi|) \,{\widetilde\omega^k\over \vol(\widetilde\Sigma_\epsilon)} \Big]  \leq \int_{\widetilde\Sigma_\epsilon}e^{\alpha''\Pi_1^*(|\phi|)}\, {\widetilde\omega^k\over \vol(\widetilde\Sigma_\epsilon)}  \leq {c''\over  \vol(\widetilde\Sigma_\epsilon)}. $$
 Therefore,
$$\int_{\widetilde\Sigma_\epsilon} \Pi_1^*(|\phi|)\, \widetilde\omega^k \leq {\vol(\widetilde\Sigma_\epsilon)\over \alpha''} \log  {c''\over  \vol(\widetilde\Sigma_\epsilon)}    \lesssim \epsilon\cdot  \log (\epsilon^{-2k})\lesssim \epsilon^{1/2}.$$
This finishes the proof of Claim 2.
\endproof

Using the above Claims and taking $\epsilon:=\norm{\phi}_{L^1}^{2/(2\alpha+1)}$, we get from \eqref{pi-l1} that $\|\Pi_1^*(\phi)\|_{L^1}\lesssim \norm{\phi}_{L^1}^{1/(2\alpha+1)}$ and finish the proof of Lemma \ref{theta-holder}.
\endproof

\begin{lemma}
	Let $\epsilon>0$ be any positive number. Then there is a constant $c=c(\epsilon)>0$ such that for any $n\geq 0$ and any positive closed $(1,1)$-current $S$ of mass $1$ we have
	$$\|(f^n)_*(S)\|\leq c(d_{k-1}+\epsilon)^n.$$
\end{lemma}
\proof
We fix a norm on the cohomology group $H^*(X,\C)$ of $X$. Recall from introduction that 
$$\lim_{n\to\infty} \|(f^n)^*: H^{k-1,k-1}(X,\C) \to H^{k-1,k-1}(X,\C)\|^{1/n} = d_{k-1}.$$
By Poincar\'e duality, we get
$$\lim_{n\to\infty} \|(f^n)_*: H^{1,1}(X,\C) \to H^{1,1}(X,\C)\|^{1/n} = d_{k-1}.$$

By \cite{Demailly-regular,dinh-sibony-regular}, there are positive closed smooth $(1,1)$-forms $S_j^\pm$ on $X$ such that $\|S_j^\pm\|$ is bounded by a constant and 
$S_j^+-S_j^-\to S$ in the sense of currents. Since $(f^n)_*$ acts continuously on positive closed $(1,1)$-currents, we can replace $S$ by $S_j^\pm$ in order to assume that $S$ is smooth. 

Since $\|S\|=1$, its cohomology class in $H^{1,1}(X,\C)$ has a norm bounded by a constant $A$. Thus, the norm of the cohomology class of $(f^n)_*(S)$ in $H^{1,1}(X,\C)$ is bounded by
$$A\|(f^n)_*: H^{1,1}(X,\C) \to H^{1,1}(X,\C)\|.$$
We easily deduce that $\limsup_{n\to \infty} \|(f^n)_*(S)\|^{1/n}\leq d_{k-1}$  using the fact that 
the mass of a positive closed current only depends on its cohomology class. The lemma follows.
\endproof

We continue the proof of Theorem \ref{main-theorem}. Recall that $d_{k-1}<d_k$. According to the last lemma, replacing $f$ by an iterate (this doesn't change the measure $\mu$) allows us to assume that 
$$\kappa:=\sup \big\{\|f_*(S)\|:\, S \text{ positive closed $(1,1)$-current of mass 1}\big\} <d_k.$$
We fix a $\delta$ with $1<\delta<d_k/\kappa$ and   define $\Lambda:=\delta d_k^{-1} f_*$ acting on $\DSH(X)$. 
\vskip 5pt

 Recall from \cite{dinh-sibony-CMH} that $\phi\mapsto \int_X \phi\,d\mu$ is continuous on $\DSH(X)$. So we can replace the $\DSH$-norm by the following equivalent norm
$$\|\phi\|_\mu:=\Big|\int_X \phi \,d\mu\Big| + \inf\|T^\pm\|,$$
where the infimum is taken over all positive closed $(1,1)$-currents $T^\pm$  such that $\ddc \phi =T^+-T^-$.

Consider the hyperplane of $\DSH(X)$ defined by
$$E:=\Big\{\phi\in\DSH(X):\, \int_X \phi\,d\mu=0 \Big\}.$$

\begin{lemma}\label{3.4}
The space $E$ is invariant by $\Lambda$ and  we have $\|\Lambda\|_\mu\leq 1$. In particular, any ball of center $0$ in  $E$ is invariant by $\Lambda$.
\end{lemma}
\proof
Since $\mu$ has no mass on pluripolar sets and $f^*(\mu)=d_k \mu$,  we have for $\phi\in E$,
$$\int_X \Lambda(\phi)\,d\mu = \int_X \delta d_k^{-1}f_*(\phi) \,d\mu =\delta \int_X  \phi\, d_k^{-1} \,df^*(\mu)=\delta  \int_X  \phi\,d\mu=0. $$
 So $E$ is invariant by $\Lambda$.

For the second assertion, it is enough to prove $\norm{\Lambda(T)}\leq \norm{T}$ for all positive closed $(1,1)$-current $T$. This  follows from the assumption that $\delta d_{k}^{-1}\kappa<1$.
\endproof

Now we can complete the proof of Theorem \ref{main-theorem}.

\begin{proof}[End of the proof of Theorem \ref{main-theorem}]
For $\phi,\psi\in\DSH(X)$ with the DSH-norm bounded by a fixed constant $M>1$, define $\phi':=d_k^{-1}f_*(\phi)-\phi$ and $\psi':=d_k^{-1}f_*(\psi)-\psi$. Then both $\phi'$ and $\psi'$ are in $E$ since $\mu$ is invariant by $d_k^{-1}f^*$. 

The  definition of $\phi'$ gives
\begin{align*}
 \sum_{n=0}^\infty  \int_X   d_k^{-n}(f^n)_*(\phi') \,\omega^k &= \sum_{n=0}^\infty  \int_X  d_k^{-n}(f^n)_*(d_k^{-1}f_*(\phi)-\phi) \,\omega^k \\
 &=\sum_{n=0}^\infty \Big[  \int_X   d_k^{-(n+1)}(f^{n+1})_*(\phi) \,\omega^k -   \int_X   d_k^{-n}(f^{n})_*(\phi)   \,\omega^k   \Big]\\
 &=\lim_{n\to \infty}   \int_X d_k^{-n}(f^n)_*(\phi) \,\omega^k -\int_X \phi \,\omega^k.
 \end{align*}
 Recall that $d_k^{-n}(f^n)^*(\omega^k)$ converges to $\mu$. So
 $$\lim_{n\to \infty} \int_X  d_k^{-n}(f^n)_*(\phi) \,\omega^k = \lim_{n\to \infty} \int_X \phi \, d_k^{-n}(f^n)^*(\omega^k)=\int_X \phi \, d\mu.$$
 Therefore, by the definition of $\Lambda$, we have 
 $$\int_X \phi \, d\mu=\int_X \phi \,\omega^k+ \sum_{n=0}^\infty  \int_X   d_k^{-n}(f^n)_*(\phi') \,\omega^k =\int_X \phi \,\omega^k + \sum_{n=0}^\infty  \delta^{-n} \int_X  \Lambda^n(\phi') \,\omega^k.$$
 The same equality holds after replacing $\phi,\phi'$ by $\psi,\psi'$.
 Hence 
$$
 \int_X \phi \,d\mu -\int_X \psi \,d\mu  =\int_X  (\phi-\psi)\, \omega^k+ \sum_{n=0}^\infty \delta^{-n} \int_X   \Lambda^n(\phi' -\psi') \,\omega^k.
$$

Notice that $\norm{\phi'}_{\DSH},\norm{\psi'}_{\DSH}$ are both bounded by a fixed constant times $M$. So $\phi'$ and $\psi'$ belong to a ball $\B$ (with respect to the $L^1$-norm) of center $0$ in $E$ with  radius  bounded by a constant times $M$. Moreover, by Lemmas \ref{theta-holder} and \ref{3.4}, $\Lambda$ is an $\theta$-H\"older map from $\B$ to $\B$ with respect to the $L^1$-distance for some $0<\theta\leq 1$. 
So we can apply  Lemma \ref{log-continuous} with $u(\phi):=\int_X \phi\, \omega^k$, getting
$$\Big|\sum_{n=0}^\infty \delta^{-n}  \int_X  \Lambda(\phi' -\psi') \,\omega^k \Big|\lesssim \big(\log^\star\norm{\phi'-\psi'}_{L^1}\big)^{-\gamma}$$
for some $\gamma>0$.

By using Lemma  \ref{theta-holder} again, we have
$$\norm{\phi'-\psi'}_{L^1}\leq \big\|d_k^{-1}f_*(\phi-\psi)  \big\|_{L^1} +\norm{\phi-\psi}_{L^1}\lesssim \norm{\phi-\psi}_{L^1}^\theta+  \norm{\phi-\psi}_{L^1} \lesssim  \norm{\phi-\psi}_{L^1}^\theta.$$
Here, we use that $\norm{\phi-\psi}_{L^1}\lesssim M$.

We conclude that 
$$
\Big| \int_X \phi \,d\mu -\int_X \psi \,d\mu \Big| \lesssim \norm{\phi-\psi}_{L^1}+\big(\log^\star\norm{\phi-\psi}_{L^1}^\theta\big)^{-\gamma}. 
$$

Since $\norm{\phi-\psi}_{L^1}\lesssim M$, to prove the desired inequality, it is enough to consider the case where $t:=\norm{\phi-\psi}_{L^1}\leq 1/2$. In this case, we have
$$t  \lesssim  (  1-\log t  )^{-\gamma}$$
and 
$$ (  1-\theta\log t )^{-\gamma} = \theta^{-\gamma}(  \theta^{-1}-\log t )^{-\gamma}  \leq  \theta^{-\gamma} (1-\log t)^{-\gamma}.   $$
We deduce that 
$$
\Big| \int_X \phi \,d\mu -\int_X \psi \,d\mu \Big| \leq c (1-\log t)^{-\gamma} =c( \log^\star t)^{-\gamma} $$
for some constant $c>0$ depending on $\theta$ and $\gamma$.
The proof of Theorem \ref{main-theorem} is now complete.
\end{proof}

\end{document}